\documentclass[a4paper,11pt]{article}
\usepackage{enumerate,amsmath,amsfonts,bm,float}
\usepackage{latexsym,amssymb,amsthm, amsfonts, euscript}
\usepackage{amscd,enumerate,verbatim,calc}
\usepackage{longtable, scalefnt, mathrsfs}
\newtheorem{Exaexa}{Example}[section]

\numberwithin{equation}{section}
\newtheorem{thm}{Theorem}[section]
\theoremstyle{definition}

\newtheorem{rem}{Remark}[section]
\theoremstyle{plain}
\newtheorem{lem}[thm]{Lemma}

\newtheorem{cor}{Corollary}[section]

\begin{document}

\begin{center}{\large{\bf{ Some subgroups of a finite field and their applications for obtaining explicit factors}}}\end{center}
\begin{center} Manjit Singh \\
 Email: manjitsingh.math@gmail.com\\
Department of Mathematics, Deenbandhu Chhotu Ram University of Science and Technology, Murthal-131039, India\end{center}
\begin{abstract}   Let $\mathcal{S}_q$ denote the group  of all square elements in  the multiplicative group $\mathbb{F}_q^*$ of  a finite field $\mathbb{F}_q$  of odd characteristic containing  $q$ elements. Let $\mathcal{O}_q$ be the set  of  all odd order elements of $\mathbb{F}_q^*$.  Then $\mathcal{O}_q$ turns up as a subgroup of $\mathcal{S}_q$. In this paper, we show that $\mathcal{O}_q=\langle4\rangle$ if $q=2t+1$ and, $\mathcal{O}_q=\langle t\rangle $ if $q=4t+1$, where   $q$ and $t$ are odd primes.  This paper  also gives a direct method for obtaining the  coefficients of irreducible factors of $x^{2^nt}-1$  in $\mathbb{F}_q[x]$  using the information of generator elements of $\mathcal{S}_q$  and $\mathcal{O}_q$, when $q$ and $t$ are odd primes such that  $q=2t+1$ or $q=4t+1$. \end{abstract}
\section{Introduction}
 Factoring polynomials over finite fields plays an important role in  algebraic coding theory for the error-free transmission of information and cryptology for the secure transmission of information. Please note that the availability of explicit factors of $x^m-1$ over finite fields, especially irreducible polynomials over finite fields is useful for analyzing the structure and inner-relationship of codewords of a code and other areas of electrical engineering where linear feedback shift registers (LFSR) are used (see \cite{ber, gol, lidl}).  For example, factoring of $x^{m}-1$  into irreducible factors over finite fields  is essentially useful to describe the theory of cyclic codes of length $m$ over finite fields. 

Blake, Gao and Mullin \cite{bla} explicitly determined all the irreducible factors of  $x^{2^n}\pm1$ over $\mathbb{F}_p$, where $p$ is a prime with $p\equiv3\pmod 4$.  Chen, Li and Tuerhong  \cite{chen} gave the explicit factorization of $x^{2^m p^n}-1$ over $\mathbb{F}_q$, where $p$ is an odd prime with $q \equiv1\pmod p$. In \cite{bro}, Brochero Mart\'{i}nez,  Giraldo Vergara and de Oliveira  generalized the results in \cite{chen} by giving the explicit factorization of $x^{m}-1$  over $\mathbb{F}_q$, where every prime factor of $m$ divides $q-1$. Meyn \cite{meyn} obtained the irreducible factors of cyclotomic polynomials $\Phi_{2^k}(x)$ over $\mathbb{F}_q$ when  $q\equiv3\pmod 4$. Fitzgerald and Yucas \cite{fitz7} obtained the explicit factorization of $\Phi_{2^k3}(x)$ when $q\equiv\pm1\pmod 3$. In \cite{ww},  Wang and Wang obtained the explicit factorization of $\Phi_{2^k5}(x)$ for  $q\equiv \pm 2\pmod 5$. When $q$ and  $r$ are distinct odd primes,  Stein \cite{stein}  computed the factors of $\Phi_{r}(x)$  from the traces of the roots of $\Phi_r(x)$ over  prime field $\mathbb{F}_q$.  Assuming that the explicit factors of $\Phi_r(x)$ are known, Tuxanidy and Wang  \cite{tw} obtained the irreducible factors of $\Phi_{2^nr}(x)$ over $\mathbb{F}_q$, where $r>1$ is an arbitrary odd integer.

 In this paper,  the explicit factorization of $\Phi_{2^nd}(x)$  over $\mathbb{F}_q$ is revisited for every odd prime divisor $d$ of $q-1$.   Using this factorization, when $q$ and $t$ are primes with $q=2lt+1$ for $l=1, 2$,  the  explicit factors of $\Phi_{2^nt}(x)$  over $\mathbb{F}_q$ are determined. This factorization provides the complete information regarding the coefficients of irreducible factors of   of  $x^{2^nt}-1$ over $\mathbb{F}_q$.    

 The  paper is organized as follows:  The necessary notations and some known results to be used throughout the paper are provided in Section 2.  In Section 3,  assuming $d$ is an odd divisor of $q-1$,   the explicit factorization of $x^{2^nd}-1$ over $\mathbb{F}_q$ is reformulated in two different cases when  $q\equiv1\pmod{4}$ in Theorem \ref{fact2nd4k+1} and, when $q\equiv3\pmod 4$ in Theorem \ref{fact2t+1}. In end of Section 3, the explicit factorization of cyclotomic polynomial $\Phi_{2^nd}(x)$ over $\mathbb{F}_q$ is obtained for every odd prime $d$ with $d|(q-1)$. In  Section 4, we record few results concerning the square elements in the multiplicative group $\mathbb{F}_q^*$ which appears to be new. Using these results,   the coefficients of  irreducible factors  of $x^{2^nt}-1$ over $\mathbb{F}_q$ are obtained effortlessly when $q$ and $t$ are primes with $q=2lt+1$ for $l=1$ or $2$.	 In Section 5, in order to illustrate our results  obtained in  Section 4, we find  the factorization of $x^{2^{173\cdot2^n}}-1\in\mathbb{F}_{347}[x]$,  $x^{{704}}-1\in\mathbb{F}_{23}[x]$, $x^{2^{37\cdot2^n}}-1\in\mathbb{F}_{149}[x]$ and  $x^{2^{13\cdot2^n}}-1\in\mathbb{F}_{53}[x]$. As a consequence to our factorization, we obtain infinite families of binomials and trinomials over finite fields.
%%%%%%%%%%%%%%%%%%%%%%  
\section{The cyclotomic factorization of $x^{2^nd}-1$ over finite fields}
%%%%%%%%%%%%%%%%%%%%%%
 It is well-known that for any integer $n\ge1$,  the {\it{cyclotomic decomposition}} of $x^n-1$ is given by  \begin{center}$\displaystyle{x^n-1=\prod_{k|n}\Phi_k(x)}$;\quad $\displaystyle{\Phi_k(x)=\prod_{\substack{\gcd(i,k)=1\\0\le i\le k}}(x-\xi^i)}$,\end{center}    
where  $\xi$ is a primitive $k$th root of unity in some extension field of $\mathbb{F}_q$ and $\Phi_k(x)$ is the $k$th cyclotomic polynomial.   The degree of $\Phi_k(x)$ is $\phi(k)$, where $\phi(k)$ is the Euler Totient function. Let $e$ be the least positive integer such that $q^e\equiv1\pmod n$. Then, in $\mathbb{F}_q[x]$, $\Phi_n(x)$ splits into the product of $\phi(n)/e$ monic irreducible polynomials of degree $e$. In particular, $\Phi_n(x)$ is irreducible over $\mathbb{F}_q$ if and only if $e=\phi(n)$. Note that $\Phi_n(x)$ is irreducible over $\mathbb{F}_q$, then  $\Phi_m(x)$ is also irreducible over $\mathbb{F}_q$ for every $m|n$ (see \cite{lidl,roman}). 

\begin{lem}[see Theorem 10.7 in \cite{wan} and Theorem 3.75 in \cite{lidl}]\label{birr} Let $l\geq 2$ be an integer and $a\in \mathbb{F}_q^*$ such that the order of $a$ is $k\ge2$. Then the binomial $x^l-a\in \mathbb{F}_q[x]$  is irreducible over $\mathbb{F}_q$ if and only if the following conditions are satisfied:\begin{enumerate}
\item[(i)] Every prime factor of $l$ divides $k$, but not $(q-1)/k$;
\item[(ii)] If $4|l$, then $4|(q-1)$.\end{enumerate}\end{lem}
\begin{lem}[Theorem 10.15 in \cite{wan}] \label{tirr}Let $f(x)$ be any irreducible polynomial over $ \mathbb{F}_q$ of degree $l\geq1$. Suppose that $f(0)\neq0$ and $f(x)$ is of order $e$ which is equal to the order of any root of $f(x)$.  Let $k$ be a positive integer, then the polynomial $f(x^{k})$  is irreducible over $\mathbb{F}_q$ if and only if the following three conditions are satisfied:\begin{enumerate}
\item[(i)] Every prime divisor  of $k$ divides $e$;
\item[(ii)] gcd$(k, \frac{q^l-1}{e})=1$;
\item[(iii)]  If $4|k$, then $4|(q^l-1)$.\end{enumerate} \end{lem}

\begin{lem}\label{cyclotomic} Suppose that  $t$ is an odd prime such that $\gcd(2t,q)=1$. Then in $\mathbb{F}_q[x]$ the following properties of cyclotomic polynomials hold:\begin{enumerate}
\item[(i)] $\displaystyle{\Phi_{2^kt}(x)=\frac{\Phi_{2^k}(x^{t})}{\Phi_{2^k}(x)}}$,
\item[(ii)]  $\displaystyle{\Phi_{2^{k+r}}(x)=\Phi_{2^k}(x^{2^r})}$ for integers $k\ge1$ and $r\ge0$.
\item[(iii)] $\displaystyle{\Phi_{2^{n}t}(x)=\frac{\Phi_{2^k}(x^{2^{n-k}t})}{\Phi_{2^k}(x^{2^{n-k}})}}$ for all integer $n\ge k\ge1$. \end{enumerate} \end{lem}
\begin{proof} First and second part are given in \cite[Exercise 2.57]{lidl}. The third part is an immediate consequence of the parts (i) and (ii). \end{proof}
 Hereafter, let $\mathbb{F}_q$ be a finite field with $q=2^st+1$ for some integers $s\ge1$ and $t$ is odd. Let $\alpha_{2^k}$ be a primitive $2^k$th root of unity of $\mathbb{F}_q^*$, where $0\le k\le s$.  Then,
for any integer $n\ge1$, we present, without proof,  the well known factorization of $x^{2^n}-1$ over  $\mathbb{F}_q$  in the following lemma.
\begin{lem}\label{cyclo2n}   For any integer $n\ge 1$, the cyclotomic  factorization of $x^{2^n}-1$ over $\mathbb{F}_q$ is given by: \begin{eqnarray*} 
x^{2^n}-1=\left\{ \begin{array}{lcl}
\displaystyle{(x-1)\prod_{k=1}^{n}\Phi_{2^k}(x)}& \mbox{for}
& 1\le n\le s \\ \displaystyle{(x-1)\prod_{k=1}^{s}\Phi_{2^k}(x)\prod_{r=1}^{n-s}\Phi_{2^{s}}(x^{2^{r}})} & \mbox{for} & n>s\ge1
\end{array}\right.\end{eqnarray*} where factors  $\Phi_{2^k}(x)$ for $1\le k\le s$ and $\Phi_{2^s}(x^{2^r})$ for $1\le r\le n-s$ can be factors  as:\begin{center} $\displaystyle{\Phi_{2^k}(x)=\prod_{1\le i\le 2^{k-1}}(x-\alpha_{2^k}^{2i-1})}$  and $\displaystyle{\Phi_{2^s}(x^{2^r})=\prod_{1\le i\le 2^{s-1}}(x^{2^r}-\alpha_{2^s}^{2i-1})}$.\end{center}  \end{lem}
 The above lemma immediately gives the following:

\begin{lem}\label{fact2ndcyclo} For any integer $n\ge1$ and odd integer $d$, the factorization of $x^{2^nd}-1$ into   decomposable cyclotomic polynomials  over $\mathbb{F}_q$  is
\begin{eqnarray*} \label{gencyclo}
x^{2^nd}-1=\left\{ \begin{array}{lcl}
\displaystyle{(x^d-1)\prod_{k=1}^{n}\Phi_{2^k}(x^d)}& \mbox{for}
& 1\le n\le s \\ \displaystyle{(x^d-1)\prod_{k=1}^{s}\Phi_{2^k}(x)\prod_{r=1}^{n-s}\Phi_{2^{s}}(x^{2^{r}d})} & \mbox{for} & n>s\ge1
\end{array}\right.\end{eqnarray*} where factors $\Phi_{2^k}(x^d)$ for $1\le k\le s$ and $\Phi_{2^s}(x^{2^rd})$ for $0\le r\le n-s$  can be factors as:\begin{center} $\displaystyle{\Phi_{2^k}(x^d)=\prod_{1\le i\le 2^{k-1}}(x^d-\alpha_{2^k}^{2i-1})}$  and $\displaystyle{\Phi_{2^s}(x^{2^rd})=\prod_{1\le i\le 2^{s-1}}(x^{2^rd}-\alpha_{2^s}^{2i-1})}$.\end{center} 
\end{lem}
\begin{lem} \label{factm}For any integer $m$ relatively prime with $q$, let $b$ be a primitive $m$th root of unity in some extension field  of $\mathbb{F}_q$. Then  \begin{center}$x^{m}-1=\prod_{j=0}^{m-1}(x-b^j)$.  \end{center}Further, if $c\in\mathbb{F}_q^*$ such that $c=a^m$ for some $a\in \mathbb{F}_q^*$, then  \begin{center}$\displaystyle{x^{m}-c=\prod_{j=0}^{m-1}(x-ab^j)}$.\end{center} \end{lem}

%%%%%%%%%%%%%%%%%%%%%%
\section{Factorization of $x^{2^nd}-1$ over $\mathbb{F}_q$, when $q\equiv1\pmod {2d}$}
%%%%%%%%%%%%%%%%%%%%%%
In this section, we reformulate the factorization of $x^{2^nd}-1$ into irreducible factors over $\mathbb{F}_q$ recursively when $d$ is an odd divisor of $q-1$.  In view of Lemma \ref{birr},  each factor of $\Phi_{2^k}(x^d)$ and   $\Phi_{2^s}(x^{2^rd})$ is reducible over $\mathbb{F}_q$.  Thus, in order  to determine the complete factorization of  $x^{2^nd}-1$ over $\mathbb{F}_q$, one needs  to split the decomposable cyclotomic polynomials $\Phi_{2^k}(x^{d})$ for $1\le k\le s$ and $\Phi_{2^s}(x^{2^rd})$ for $1\le r\le n-s$ into irreducible factors over $\mathbb{F}_q$.
\begin{thm}\label{cyclo2rd} Let $d$ be an odd integer such that $q\equiv1\pmod {2^kd}$, where $1\le k\le s$.   Let  $\gamma$ be a primitive $d$th root of unity in $\mathbb{F}_q^*$. Then, for any integer $r\ge0$, the complete factorization of $\Phi_{2^k}(x^{2^rd})$ is  $$\Phi_{2^k}(x^{2^rd})=\Phi_{2^k}(x^{2^r})\prod_{\substack{1\le i\le 2^{k-1}\\1\le j\le d-1}}(x^{2^r}-\alpha_{2^k}^{2i-1}\gamma^j),$$ where   \begin{eqnarray*}
\Phi_{2^k}(x^{2^r})=\Phi_{2^{k+r}}(x)=\left\{ \begin{array}{lcl}
\displaystyle{\prod_{i=1}^{2^{k+r-1}}(x-
\alpha_{2^{k+r}}^{2i-1})} &\mbox{if} &k+r\le s \\ \displaystyle{\prod_{i=1}^{2^{s-1}}(x-
\alpha_{2^{k+r-s}}^{2i-1})} &\mbox{if} &k+r>s.
\end{array}\right.\end{eqnarray*} 
  \end{thm}
\begin{proof}For any integer $r\ge0$ and $1\le k\le s$,  observe that\begin{eqnarray*}\Phi_{2^k}(x^{2^rd})=\prod_{\substack{1\le i\le 2^{k-1}}}\big(x^{2^rd}-\alpha_{2^k}^{2i-1}\big)=\prod_{\substack{1\le i\le 2^{k-1}}}\big((x^{2^r})^{d}-\alpha_{2^k}^{d(2i-1)}\big).\end{eqnarray*} Let  $\gamma$ be a primitive $d$th root of unity in $\mathbb{F}_q^*$.  Then, by Lemma \ref{factm}    \begin{eqnarray*}\Phi_{2^k}(x^{2^rd})&=&\prod_{\substack{1\le i\le 2^{k-1}\\0\le j\le d-1}}(x^{2^r}-\alpha_{2^k}^{2i-1}\gamma^j)\\&=&\Phi_{2^k}(x^{2^r})\prod_{\substack{1\le i\le 2^{k-1}\\1\le j\le d-1}}(x^{2^r}-\alpha_{2^k}^{2i-1}\gamma^j).\end{eqnarray*} This completes the proof. \end{proof}
 \begin{thm}\label{fact2nd4k+1} Let $d$ be any odd integer and $q\equiv 1\pmod  {2d}$. Then, for any integer $n\ge1$, the factorization of $x^{2^nd}-1$  over $\mathbb{F}_q$  is  given as: 
\begin{eqnarray*}
x^{2^nd}-1=\left\{ \begin{array}{lcl}
\displaystyle{\prod_{j=0}^{d-1}
\bigg((x-\gamma^j)\prod_{\substack{i=1\\1\le k\le n}}^{ 2^{k-1}}(x-\alpha_{2^k}^{2i-1}\gamma^j)\bigg)} &\mbox{if} &n\le s \\ \displaystyle{\prod_{j=0}^{d-1}\bigg((x-\gamma^j)
\prod_{\substack{i=1\\1\le k\le s}}^{2^{k-1}}(x-\alpha_{2^k}^{2i-1}\gamma^j)\prod_{\substack{i=1\\1\le r\le n-s}}^{2^{s-1}}(x^{2^r}-\alpha_{2^s}^{2i-1}\gamma^j)\bigg)} &\mbox{if} &n>s
\end{array}\right.\end{eqnarray*} 
Further, if  $n>s\ge2$, the factorization  $x^{2^nd}-1$ over $\mathbb{F}_q$ has $2^{s-1}(n-s+2)d$ irreducible factors, however if $q\equiv3\pmod 4$, all nonlinear factors in the factorization are reducible over $\mathbb{F}_q$ except binomials $x^2+\gamma^j$ for all $0\le j\le d-1$.\end{thm}
\begin{proof} In  Theorem \ref{cyclo2rd},  on substituting  $r=0$ and $k=s$ in the polynomial  $\Phi_{2^k}(x^{2^rd})$, we obtain $\Phi_{2^k}(x^{d})=\Phi_{2^k}(x)\prod_{\substack{1\le i\le 2^{k-1}\\1\le j\le d-1}}(x-\alpha_{2^k}^{2i-1}\gamma^j)$  and $\Phi_{2^s}(x^{2^rd})=\Phi_{2^s}(x^{2^r})\prod_{\substack{1\le i\le 2^{s-1}\\1\le j\le d-1}}(x^{2^r}-\alpha_{2^s}^{2i-1}\gamma^j)$ respectively. 
 The result now follows from  Lemma \ref{fact2ndcyclo}. Further,  when $n>s\ge2$, the irreducibility of its nonlinear factors can be proved by Lemma \ref{birr}. For  $q\equiv3\pmod 4$ i.e.  $s=1$, the factorization of $x^{2^nd}-1$ over $\mathbb{F}_q$ reduces to  \begin{eqnarray*} x^{2^nd}-1=\prod_{j=0}^{d-1}\bigg((x-\gamma^j)(x+b^j)\prod_{\substack{1\le r\le n-1}}(x^{2^r}+\gamma^j)\bigg). \end{eqnarray*} By Lemma \ref{birr}, factors $x^{2^r}+\gamma^j$ are reducible over $\mathbb{F}_q$ for every $r\ge2$.\end{proof}
Consider the case   $q\equiv3\pmod 4$. Let $Q=q^2=2^uv+1$,  $u\ge3$ and $2\nmid v$. Let $\beta_{2^k}$ be a primitive $2^k$th root of unity in $\mathbb{F}_Q^*$. Note that  $\beta_{2^k}:=\alpha_{2^k}$ when $\beta_{2^k}\in\mathbb{F}_q$.
\begin{itemize}\item[(i)] A quadratic character $\chi$ on $\langle\beta_{2^u}\rangle\subseteq \mathbb{F}_Q^*$ is defined as \begin{center}
 $\chi(\beta_{2^k})=\beta_{2^k}^{q+1}=\left\{\begin{array}{ll}1 &\mbox{ if  $0\leq k<u$}\\-1 &\mbox{ if  $k=u$}\end{array} \right.$\end{center}
\item[(ii)] A trace is a mapping $\mathbb{T}:\mathbb{F}_Q\rightarrow\mathbb{F}_q$ defined as  $\mathbb{T}(x)=x+x^q $ for all $x\in\mathbb{F}_Q$. Further, for any positive integer $r\ge1$, we define the $r$th trace $\mathbb{T}_r:\mathbb{F}_Q\rightarrow\mathbb{F}_q$ such that  $\mathbb{T}_r(x)=\mathbb{T}(x^r)$.\end{itemize}
\begin{lem} [Lemma 2.6 in \cite{singh}] \label{cyclo2k}  For any fixed $3\le k\le u$. The cyclotomic polynomial $\Phi_{2^k}(x)=x^{2^{k-1}}+1$ over $\mathbb{F}_q$ can be splits into irreducible factors as  $$\Phi_{2^k}(x)=\prod_{1\le i\le 2^{k-3}} (x^2\pm\mathbb{T}(\beta_k^{2i-1})x+\chi(\beta_k)).$$\end{lem}

\begin{lem}  [Theorem 3.3 in \cite{singh}] \label{trace} If $q\equiv3\pmod4$ and $3\le k\le u$. Then there are $2^{k-2}$ distinct traces $\mathbb{T}(\beta_{2^k}^{2i-1})$ such that the first $2^{k-3}$ traces are given by the linear recursive sequence  $\mathbb{T}_{2i-1}(\beta_{2^k})=\mathbb{T}(\beta_{2^k})
\mathbb{T}(\beta_{2^{k-1}}^{i-1})
-\chi(\beta_{2^k})\mathbb{T}(\beta_{2^k}^{2i-3})$ and the rest of $2^{k-3}$ are $-\mathbb{T}(\beta_{2^k}^{2i-1})$. The initial terms of the sequence are $\mathbb{T}(\beta_4)=0$ and for $3\le k\le u$, $\mathbb{T}(\beta_{2^k})=(\mathbb{T}(\beta_{2^{k-1}})+2\chi(\beta_{2^k}))^{(t+1)/2}$.\end{lem}
The following result is a useful tool for proving  our next theorem. The empty product assumed to be 1.
\begin{thm}\label{cyclotri} Assume that $q\equiv3\pmod 4$ and $q\equiv1\pmod d$. Then $\Phi_{4}(x^{d})=x^{2d}+1=\prod_{0\le j\le d-1}(x^{2}+\gamma^{j})$ and for $3\le k\le u$,  the irreducible factorization  of decomposable cyclotomic polynomial $\Phi_{2^k}(x^{d})$   over $\mathbb{F}_q$  is  given by: $$\Phi_{2^k}(x^{d})=\Phi_{2^k}(x)\prod_{\substack{1\le i\le 2^{k-3}\\1\le j\le d-1}} (x^2\pm \gamma^j\mathbb{T}(\beta_{2^k}^{2i-1})x+\chi(\beta_{2^k})\gamma^{2j}). $$  Further, for any integer $r\ge1$ and $3\le k\le u$, the  factorization of decomposable cyclotomic polynomial $\Phi_{2^k}(x^{2^rd})$   over $\mathbb{F}_q$  is given by:   $$\Phi_{2^k}(x^{2^rd})=\Phi_{2^k}(x^{2^r})\prod_{\substack{1\le i\le 2^{k-3}\\1\le j\le d-1}} (x^{2^{r+1}}\pm \gamma^j\mathbb{T}(\beta_{2^k}^{2i-1})x^{2^r}+\chi(\beta_{2^k})\gamma^{2j}).$$ Furthermore, the decomposable polynomial $\Phi_{2^u}(x^{2^rd})$ is a product of $2^{u-2}$ irreducible trinomials over $\mathbb{F}_q$, while  the decomposable polynomial $\Phi_{2^k}(x^{2^rd})$, where $2\le k\le u-1$, is a product of $2^{k-2}$ reducible trinomials over $\mathbb{F}_q$. \end{thm}
\begin{proof}   Since $q$ is odd prime power, so $q^2\equiv1\pmod 4$, i.e., $Q\equiv1\pmod 4$. Replacing $q$ by $Q$ and $\alpha_{2^k}$ by $\beta_{2^k}$ in the result of  Theorem \ref{cyclo2rd}, we obtain the factorization of $\Phi_{2^k}(x^{d})$ over $\mathbb{F}_Q$ such as $$\Phi_{2^k}(x^{d})=\Phi_{2^k}(x)\prod_{\substack{1\le i\le 2^{k-1}\\1\le j\le d-1}}(x-\beta_{2^k}^{2i-1}\gamma^j)$$ where integer $1\le k\le u$.  
Clearly, $\Phi_{2}(x^{d})=x^d+1=(x+1)\prod_{\substack{1\le j\le d-1}}(x+\gamma^j)$ and  $\Phi_{4}(x^{d})=x^{2d}+1=\prod_{\substack{0\le j\le d-1}}(x^2+\gamma^{j})=(x^2+1)\prod_{\substack{1\le j\le d-1}}(x^2+\gamma^{j})$. Further,  for $3\le k\le u$, we can write \begin{eqnarray*}\Phi_{2^k}(x^{d})&=&\Phi_{2^{k-2}}(x^{4d})\\&=&\prod_{\substack{1\le i\le 2^{k-3}\\0\le j\le d-1}}(x^4-\beta_{2^{k-2}}^{2i-1}\gamma^j)\\&=&\prod_{\substack{1\le i\le 2^{k-3}\\0\le j\le d-1}}(x-\beta_{2^{k}}^{2i-1}\gamma^j)(x+\beta_{2^{k}}^{2i-1}\gamma^j)(x^2+\beta_{2^{k-1}}^{2i-1}\gamma^j).\end{eqnarray*}
 For any fixed $0\le j\le d-1$, using the permutation $i\mapsto 2^{k-3}-i+1$ on the set  of integers $1\le i\le 2^{k-3}$, we obtain \begin{eqnarray*}\prod_{i=1}^{2^{k-3}}(x^2+\beta_{2^{k-1}}^{2i-1}\gamma^j)&=&
\prod_{i=1}^{2^{k-3}}(x^2-\beta_{2^{k-1}}^{-2i+1}\gamma^j).\end{eqnarray*}
 Since $\beta_{2^k}^{2q(2i-1)}=\beta_{2^{k-1}}^{-2i+1}$, so that $x^2-\beta_{2^{k-1}}^{-2i+1}\gamma^{2j}=
(x-\beta_{2^k}^{q(2i-1)}\gamma^j)(x+\beta_{2^k}^{q(2i-1)}\gamma^j)$. Therefore  \begin{eqnarray*}
\prod_{i=1}^{2^{k-3}}(x^2-\beta_{2^{k-1}}^{-2i+1}\gamma^j)=\prod_{i=1}^{2^{k-3}}(x-\beta_{2^k}^{q(2i-1)}\gamma^j)(x+\beta_{2^k}^{q(2i-1)}\gamma^j).\end{eqnarray*}  Further,  $\beta_{2^k}^{2i-1}\gamma^j$ and $-\beta_{2^k}^{2i-1}\gamma^j$ are non-conjugate elements in $\mathbb{F}_Q\setminus\mathbb{F}_q$ for any $1\le i\le 2^{k-3}$. Therefore the minimal polynomial of $\pm\beta_{2^k}^{2i-1}\gamma^j$  is $x^2\pm\mathbb{T}(\beta_{2^k}^{2i-1}\gamma^j)x+
(\beta_{2^k}^{2i-1}\gamma^j)^{q+1}$. Note that $\mathbb{T}(\beta_{2^k}^{2i-1}\gamma^j)=\gamma^j\mathbb{T}(\beta_{2^k}^{2i-1})$ and $(\beta_{2^k}^{2i-1}\gamma^j)^{q+1}=
\gamma^{2j}\chi(\beta_{2^k}^{2i-1})=\gamma^{2j}\chi(\beta_{2^k})$ for every $1\le i\le 2^{k-3}$ and $3\le k\le u$. Thus we obtain $\Phi_{2^k}(x^{d})$ over $\mathbb{F}_q$ is \begin{eqnarray*}\Phi_{2^k}(x^{d})&=&\prod_{\substack{1\le i\le 2^{k-3}\\0\le j\le d-1}}(x^2\pm \gamma^j\mathbb{T}(\beta_{2^k}^{2i-1})x+\gamma^{2j}\chi(\beta_{2^k})).\end{eqnarray*}  Further, for  any integer $r\ge1$, using the transformation $x\rightarrow x^{2^r}$, we have \begin{eqnarray}\label{srim}\Phi_{2^k}(x^{2^rd})\nonumber&=&\prod_{\substack{1\le i\le 2^{k-3}\\0\le j\le d-1}} (x^{2^{r+1}}\pm \gamma^j\mathbb{T}(\beta_{2^k}^{2i-1})x^{2^r}+\chi(\beta_{2^k})\gamma^{2j})
\\\nonumber&=&\Phi_{2^k}(x^{2^r})\prod_{\substack{1\le i\le 2^{k-3}\\1\le j\le d-1}} (x^{2^{r+1}}\pm \gamma^j\mathbb{T}(\beta_{2^k}^{2i-1})x^{2^r}+\chi(\beta_{2^k})\gamma^{2j}).\end{eqnarray}
 Then by  Lemma \ref{tirr}, every trinomial $x^{2^{r+1}}\pm \gamma^j\mathbb{T}(\beta_{2^k}^{2i-1})x^{2^r}+\chi(\beta_{2^k})\gamma^{2j}$ is reducible for $3\le k\le u-1$ and irreducible over $\mathbb{F}_q$ for $k=u$.\end{proof}

 In the following theorem, we determine the factorization of $x^{2^nd}-1$ over $\mathbb{F}_q$, when $q\equiv3\pmod 4$ and $q\equiv1\pmod d$. 
\begin{thm}\label{fact2t+1}If $q\equiv3\pmod 4$ and $d|(q-1)$, then $x^{2^nd}-1$  can be written as a product of  $d(2^{u-2}(n-u+2)+1)$  irreducible factors over $\mathbb{F}_q$ as:  \begin{eqnarray*}
x^{2^nd}-1&=&(x^{2^n}-1)\prod_{1\le j\le d-1}(x\pm \gamma^j)\prod_{\substack{2\le k\le u-1\\1\le i\le 2^{k-2}\\1\le j\le d-1}}(x^2-\gamma^{j}\mathbb{T}(\beta_{2^k}^{2i-1})x+\gamma^{2j})\\&&
\prod_{\substack{0\le r\le n-u\\1\le i\le 2^{u-3}\\1\le j\le d-1}}(x^{2^{r+1}}\pm \gamma^{j}\mathbb{T}(\beta_{2^u}^{2i-1})x^{2^r}-\gamma^{2j}).
\end{eqnarray*} \end{thm}
\begin{proof}   By substituting $s=1$ in Lemma \ref{fact2ndcyclo}, the factorization of $x^{2^nd}-1$ over $\mathbb{F}_q$ reduces to  \begin{eqnarray*}
x^{2^nd}-1=(x^d-1)\Phi_{2}(x^{d})\prod_{1\le r\le n-1}\Phi_{2}(x^{2^{r}d}).
\end{eqnarray*}  At this point, we recall $u=\max\{r\in\mathbb{Z}:2^r|(Q-1)\}$. Then we write \begin{eqnarray*}
x^{2^nd}-1\nonumber&=&(x^{2d}-1)\prod_{k=2}^{u-1}
\Phi_{2^k}(x^{d})\prod_{r=u}^{n}\Phi_{2^r}(x^{d})\\&=&\prod_{j=0}^{d-1}\bigg((x\pm \gamma^j)(x^2+\gamma^j)\prod_{k=3}^{u-1}\Phi_{2^k}(x^{d})\prod_{r=0}^{n-u}\Phi_{2^u}(x^{2^rd})\bigg).
\end{eqnarray*}The result now follows from  Theorem \ref{cyclotri}.\end{proof}
 In the following corollary, the factorization of $\Phi_{2^{n}d}(x)$ over $\mathbb{F}_q$  is to be deduced for every prime odd divisor $d$ of $q-1$. 
\begin{cor}\label{phi2nd4k+1} Let $q$ be an odd prime power and   $d$ be an odd prime such that $d|(q-1)$. 
\begin{enumerate}
\item[(i)]  If $q\equiv1\pmod 4$ and $2\le k\le s$. Then, for any integer $n\ge k$,   the factorization of $\Phi_{2^{n}d}(x)$ over $\mathbb{F}_q$  into $2^{k-1}(d-1)$ factors is given by:
\begin{eqnarray*}
\Phi_{2^{n}d}(x)=
\displaystyle{\prod_{\substack{1\le i\le 2^{k-1}\\1\le j\le d-1}}(x^{2^{n-k}}-\alpha_{2^k}^{2i-1}\gamma^j)}.\end{eqnarray*}
All these factors of $\Phi_{2^{n}d}(x)$ are irreducible over $\mathbb{F}_q$ when $k=s$.
\item[(ii)] If $q\equiv3\pmod 4$ and $3\le k\le  u$. Then, for any $n\ge k$,  the factorization of $\Phi_{2^{n}d}(x)$ into $2^{k-2}(d-1)$ factors over $\mathbb{F}_q$ is given by: \begin{eqnarray*}
\Phi_{2^{n}d}(x)=\displaystyle{\prod_{\substack{1\le i\le 2^{k-3}\\1\le j\le d-1}} (x^{2^{n-k+1}}\pm \gamma^j\mathbb{T}(\beta_{2^k}^{2i-1})x^{2^{n-k}}+\chi(\beta_{2^k})\gamma^{2j})}.\end{eqnarray*} \end{enumerate}

All these factors of $\Phi_{2^{n}d}(x)$ are irreducible over $\mathbb{F}_q$  when $k=u$.

\end{cor}
\begin{proof} For any integer $n\ge k\ge 1$, by Lemma \ref{cyclotomic} (iii), the cyclotomic polynomial $\Phi_{2^nd}(x)$ over $\mathbb{F}_q$ is \begin{center}$\displaystyle{\Phi_{2^{n}d}(x)=\frac{\Phi_{2^k}(x^{2^{n-k}d})}{\Phi_{2^k}(x^{2^{n-k}})}}$ for integer $n\ge k\ge1$.\end{center}
 The remaining part of proof now follows from Theorem \ref{cyclo2rd} and Theorem \ref{cyclotri}. The irreducibility follows from Lemma \ref{birr} for binomials and Lemma \ref{tirr} for trinomials factors of $\Phi_{2^{n}d}(x)$ over $\mathbb{F}_q$. \end{proof}
%%%%%%%%%%%%%%%%%
  \section{Main results}
  In this section, we introduce a direct method to obtain the coefficients of irreducible factors of and $\Phi_{2^nt}(x)$ and hence of  $x^{2^nt}-1$  over $\mathbb{F}_q$ when $q$ and $t$ are  odd primes such that either $q=2t+1$ or $q=4t+1$. First, we define  $\mathcal{S}_q=\{a^2:a\in\mathbb{F}_q^*\}$ and $\mathcal{O}_q=\{a\in\mathbb{F}_q^*:O_q(a)$ is odd$\}$, where $O_q(a)$ denotes the order of $a\in\mathbb{F}_q^*$.  Note that $\mathcal{S}_3=\mathcal{O}_3=\{1\}$, $\mathcal{S}_5=\mathcal{O}_5=\{1,4\}$, $\mathcal{S}_7=\mathcal{O}_7=\{1,2,4\}$.
\begin{thm} \label{osq} For any odd prime power $q$,  $\mathcal{S}_q$ and $\mathcal{O}_q$ are subgroups of $\mathbb{F}_q^*$ such that $\mathcal{O}_q\subseteq \mathcal{S}_q$.  Further, if $q=2^st+1$ for some integer $s\ge1$ and $t$ is an odd integer. Then, the subgroup $\mathcal{O}_q$ has $t$ distinct element  and the set $\mathcal{S}_q\setminus \mathcal{O}_q$ contains $(2^{s-1}-1)t$ elements of $\mathcal{S}_q$. Further, $\mathcal{O}_q=\mathcal{S}_q$ if and only if $q\equiv3\pmod 4$.\end{thm} \begin{proof} Let $q=2^st+1$ with integer $s\ge1$ and $t$ is odd. Since $\mathcal{S}_q$ contains $(q-1)/2$ distinct elements of $\mathbb{F}_q^*$, so the order of $\mathcal{S}_q$, i.e., $|\mathcal{S}_q|=2^{s-1}t$. Now let $a\in\mathcal{O}_q$ with $|a|=l$, then $l$ is odd. By the converse of Lagrange's theorem, $l|(q-1)$. Since $l$ is odd, so $l|t$ and hence $a\in\mathcal{S}_q$. It follows that $\mathcal{O}_q\subseteq\mathcal{S}_q$ and $|\mathcal{O}_q|=\max\{l:|a|=l$ and $a\in \mathcal{O}_q\}=t$.   In view of the above, it is trivial to note the number of elements in  $\mathcal{S}_q\setminus\mathcal{O}_q$ is $(2^{s-1}-1)t$. Further, if $q\equiv3\pmod 4$, then $s=1$ and hence $\mathcal{O}_q=\mathcal{S}_q$. \end{proof}   
\begin{thm} \label{os} Let $q$ and $t$ be odd primes such that $q=2t+1$. Then $\mathcal{S}_q=\mathcal{O}_q=\langle 4\rangle$.\end{thm}
\begin{proof} By Theorem \ref{osq},   $\mathcal{S}_q= \mathcal{O}_q$.  Since $4\in\mathcal{S}_q$,  so $4\in\mathcal{O}_q$. Note that $\mathcal{O}_q$ is cyclic  group of prime order $t$, so any element of $\mathcal{O}_q$, except 1, works as a generator. It follows that $\mathcal{O}_q=\langle 4\rangle$.
\end{proof}
\begin{lem}[see Corollary 7.10 in \cite{jj}]\label{qr} If $p$ is an odd prime then \begin{center}$\bigg(\dfrac{2}{p}\bigg)=(-1)^{(p^2-1)/8}.$\end{center} \end{lem}
\begin{thm} \label{os1}Let $q$ and $t$ be odd primes such that $q=4t+1$. Then  $t, \sqrt{t}\in\mathcal{S}_q$. Further
\begin{enumerate}\item[(i)] $\mathcal{O}_q=\langle t\rangle$.
\item[(ii)]  $\mathcal{S}_q=\langle 4\rangle$ and $\mathcal{O}_q=\langle 16\rangle$ for $q>13$. \end{enumerate} \end{thm}
\begin{proof} Let $q=2t+1$, where $q$ and $t$ be primes.  Since  $4,-1\in\mathcal{S}_q$ and $4t=-1\in\mathbb{F}_q^*$, so that $t\in\mathcal{S}_q$ and hence $\sqrt{t}\in\mathbb{F}_q^*$.  From Lemma \ref{qr}, it follows that $2\notin\mathcal{S}_q$ as  $q\equiv5\pmod 8$. Since $2\sqrt{t}=\sqrt{-1}$ or $2\sqrt{t}=-\sqrt{-1}$ with $2$ and $\pm\sqrt{-1}$ do not belong to $\mathcal{S}_q$, so that $\sqrt{t}\in\mathcal{S}_q$ because the product of a square and non-square element always a non-square element in $\mathbb{F}_q^*$.
\begin{enumerate}
\item[(i)] In this item, we shall show that $t$ is an element of $\mathcal{O}_q$ of the order $t$, that is $O_q(t)=t$. Since $t\in\mathcal{S}_q$, so   $O_q(t)=t$ or $2t$.   On contrary assume that, $O_q(t)=2t$. This yields that  $t^{t}\equiv-1\pmod q$. Using the fact $4t\equiv-1\pmod q$ and  the arithmetic in $\mathbb{F}_q$, we have $t^{(t-1)/2}\pm2\equiv 0\pmod q$. This implies $2\in \mathcal{S}_q$ or $-2\in \mathcal{S}_q$, a contradiction.
\item[(ii)] Recall $2\notin\mathcal{S}_q$. Therefore the order of $2$ is $4t$. Using exponent rule, it follows that $O_q(4)=O_q(2^2)=\dfrac{4t}{\gcd(2,4t)}=2t$ and $O_q(4^2)=\dfrac{2t}{\gcd(2,2t)}=t$.  This completes the proof.\end{enumerate} \end{proof}
 \begin{rem}  Since $t\in\mathcal{O}_q=\langle 16\rangle$, so $t=16^i $ for some  unique integer $1\le i\le t-1$. Thus $\sqrt{t}=4^i$ and hence $\sqrt{t}\in\mathcal{S}_q$. For example taking $q=53$, then $t=13=16^6$ and $\sqrt{t}=16^3=15\in\mathcal{O}_{53}$.
\end{rem}
\begin{thm}\label{phi2nd4t+1} Let $q$  and $t$ be odd primes. 
\begin{enumerate}
\item[(i)]  If $q=4t+1$. Then, for any integer $n\ge k$,   the factorization of $\Phi_{2^{n}t}(x)$ over $\mathbb{F}_q$ is given by:
\begin{eqnarray*}
\Phi_{2^{n}t}(x)=
\displaystyle{\prod_{\substack{1\le j\le t-1}}(x^{2^{n-2}}\pm\sqrt{-1}\cdot16^j)}.\end{eqnarray*}
\item[(ii)] If $q=2t+1$ and $3\le k\le  u$. Then, for any $n\ge k$,  the factorization of $\Phi_{2^{n}t}(x)$ over $\mathbb{F}_q$ is given by: \begin{eqnarray*}
\Phi_{2^{n}t}(x)=\displaystyle{\prod_{\substack{1\le i\le 2^{k-3}\\1\le j\le t-1}} (x^{2^{n-k+1}}\pm 4^j\mathbb{T}(\beta_{2^k}^{2i-1})x^{2^{n-k}}+\chi(\beta_{2^k})16^{j})}.\end{eqnarray*} \end{enumerate}
\end{thm}
\begin{proof} Let $q$ and $t$ be odd primes.  The proof will follow by applying Theorem \ref{os} and Theorem \ref{os1} in Corollary \ref{phi2nd4k+1} in two different cases $q=4t+1$ and $q=2t+1$. \end{proof}
 \begin{rem}\label{RAM}  
 In particular,   the factorization $\Phi_{2^n\cdot5}(x)$ is  the same as in   \cite[Theorem 3.1\& 3.2]{ww})  when $q\equiv1\pmod 5$   and the factorization $\Phi_{2^n\cdot3}(x)$ is same as in \cite[Propostion 3 (see parts 1-3)]{fitz7}   when $q\equiv1\pmod 3$.
 Further,  in view of   \cite[Corollary 3.3]{ww}, the computation of the  coefficients of factors of  $\Phi_{2^n5}(x)$ requires to solve two nonlinear recurrence relations, while in our case, all coefficients can be obtained  directly using Theorem \ref{os} or Theorem \ref{os1}. \end{rem}
In the following two theorems,  we obtain the factorization of $x^{2^nt}-1$ into irreducible factors over $\mathbb{F}_q$ when  either $q=2t+1$ or $q=4t+1$. In particular, the factorization of $x^{2^nt}-1$ into irreducible factors in $\mathbb{F}_q[x]$ plays a very important role to describe cyclic codes of length $2^nt$ over $\mathbb{F}_q$. 
\begin{thm}\label{factt}Let $q$ and $t$ be odd primes such that $q=2t+1$, then  \begin{eqnarray*}
x^{2^nt}-1&=&\prod_{\substack{0\le j\le t-1}}(x\pm 4^{j})\prod_{\substack{2\le k\le u-1\\1\le i\le 2^{k-2}\\0\le j\le t-1}}(x^2-4^j\mathbb{T}(\beta_{2^k}^{2i-1})x+4^{2j})\\&&
\prod_{\substack{0\le r\le n-u\\1\le i\le 2^{u-3}\\ 0\le j\le t-1}}(x^{2^{r+1}}\pm 4^j\mathbb{T}(\beta_{2^u}^{2i-1})x^{2^r}-4^{2j}).
\end{eqnarray*}  \end{thm}
\begin{proof} The proof follows immediately by using Theorem \ref{trace}, Theorem \ref{fact2t+1} and Theorem \ref{os}.\end{proof}

\begin{thm}\label{fact2nt4t+1} Let $q$ and $t$ be odd primes such that $q=4t+1$. Then the factorization of $x^{2^nt}-1$ into the product of  $2nt$ irreducible polynomials over $\mathbb{F}_q$  is  given as: 
\begin{eqnarray*}\label{ex2nd4k+1}
x^{2^nt}-1=\prod_{j=0}^{t-1}\bigg((x\pm{16}^j)(x\pm\sqrt{-1}\cdot{16}^j)\prod_{\substack{1\le r\le n-2}}(x^{2^r}\pm\sqrt{-1}\cdot{16}^j)\bigg).
\end{eqnarray*}\end{thm}
\begin{proof} The proof follows immediately from Theorem \ref{fact2nd4k+1} and Theorem \ref{os1}.\end{proof}
%%%%%%%%%%%%%%%
\section{Worked Examples}
In this section, we give   some examples to illustrate our results. In particular, if $q$ and $t$ are primes such that $q\in\{2t+1,4t+1\}$, then all coefficients of irreducible factors can be determined directly.
 \begin{Exaexa} Let $q=347=2\cdot 173+1$. Then $s=1$,  $t=173$ and $u=3$. Now $\beta_2={-1}$,  $\mathbb{T}(\beta_4)=0$ and $\mathbb{T}(\beta_8)=\sqrt{-2}=
(-2)^{87}=107$. By Theorem \ref{os},  $4$ is a primitive $173$th root of unity in $\mathbb{F}_{347}^*$. This follows that $x^{173}-1=\prod_{j=0}^{172}(x-4^j)$ and $x^{173}+1=\prod_{j=0}^{172}(x+4^j)$. Also $x^{346}+1=\prod_{j=0}^{172}(x^2+4^j)$. Further, for $n\ge3$,  by Theorem \ref{factt}, the factorization of $x^{173\cdot 2^n}-1$ into  $173(2n-1)$  irreducible factors over $\mathbb{F}_{347}$  is given as: \begin{eqnarray*}
x^{173\cdot2^n}-1&=&\prod_{0\le j\le 172}\bigg((x\pm4^j)(x^2+4^{2j})\prod_{\substack{0\le r\le n-3}}(x^{2^{r+1}}\pm4^{j}\cdot 107x^{2^r}-4^{2j})\bigg).\end{eqnarray*}\end{Exaexa}
\begin{Exaexa}\label{23} Let $q=23=2\cdot 11+1$. Then $s=1$,  $t=11$ and $u=4$. In $\mathbb{F}_{23}^*$,  $\beta_2={-1}$,  $\mathbb{T}(\beta_4)=0$, $\mathbb{T}(\beta_8)=\sqrt{2}=
2^{6}=-5$ and $\mathbb{T}(\beta_{16})=\sqrt{-5-2}=
(-7)^{6}=4$, $\mathbb{T}_3(\beta_{16})=7$.  By Theorem \ref{os},   $4$ is a primitive $11$th root of unity in $\mathbb{F}_{11}^*$. This follows that $x^{11}-1=\prod_{j=0}^{10}(x-4^j)$ and $x^{11}+1=\prod_{j=0}^{10}(x+4^j)$. Also $x^{22}+1=\prod_{j=0}^{10}(x^2+4^j)$. Further,   by Theorem \ref{factt}, the factorization of $x^{352}-1$ into  $143$  irreducible factors over $\mathbb{F}_{23}$  is given as:\begin{eqnarray*}
x^{352}-1&=&\prod_{\substack{0\le j\le 10}}\bigg((x\pm 4^{j})\prod_{\substack{2\le k\le 3\\1\le i\le 2^{k-2}}}(x^2-4^j\mathbb{T}(\beta_{2^k}^{2i-1})x+4^{2j})\\&&
\prod_{\substack{1\le i\le 2}}(x^2\pm 4^j\mathbb{T}(\beta_{16}^{2i-1})x-4^{2j})(x^4\pm 4^j\mathbb{T}(\beta_{16}^{2i-1})x^2-4^{2j})\bigg)\\&=&
(x^{44}-1)\prod_{\substack{0\le j\le 10\\\eta\in\{4,7\}}}\bigg((x^2\pm4^j\cdot 5 x+4^{2j})\\&&\cdot(x^2\pm 4^j\eta x-4^{2j})(x^4\pm 4^j\eta x^2-4^{2j})\bigg).\end{eqnarray*} 
Further,  using recursive approach, the factorization of $x^{704}-1$ into  $187$  irreducible factors over $\mathbb{F}_{23}$  is given by \begin{eqnarray*}
x^{704}-1&=&(x^{352}-1)\prod_{\substack{0\le j\le 10\\\eta\in\{4,7\}}}(x^8\pm 4^j\eta x^4-4^{2j})\big).\end{eqnarray*}  
Note that $23^4\equiv1\pmod {352}$. By Theorem \ref{phi2nd4t+1},  the  factorization of cyclotomic polynomial $\Phi_{352}(x)$ into $40$ (i.e. $\phi(352)/4$) irreducible factors over $\mathbb{F}_{23}$ is given by: \begin{eqnarray*}
\Phi_{352}(x)&=&\displaystyle{\prod_{\substack{1\le i\le 2\\1\le j\le 10}} (x^{4}\pm 4^j\mathbb{T}(\beta_{16}^{2i-1})x^{2}-16^{j})}\\&=& \displaystyle{\prod_{j=1}^{10} (x^{4}\pm 4^j\cdot4 x^{2}-16^{j})(x^{4}\pm 4^j\cdot7 x^{2}-16^{j})}.\end{eqnarray*} 
\end{Exaexa} 
 \begin{Exaexa} Let $q=149=4\cdot 37+1$. Then $s=2$,  $t=37$.  By Theorem \ref{os1},   $\alpha_4=\sqrt{-1}=\sqrt{148}=2\sqrt{37}=2\cdot 16^9=-44$. Using Theorem \ref{fact2nt4t+1}, the factorization of $x^{2^n\cdot37}-1$ over $\mathbb{F}_{149}$ can be written as a product of  $74n$  irreducible factors as:
\begin{eqnarray*}\label{ex2nd4k+1}
x^{2^n\cdot37}-1=\prod_{j=0}^{36}\bigg((x\pm{16}^j)(x\pm44\cdot{16}^j)\prod_{\substack{1\le r\le n-2}}(x^{2^r}\pm44\cdot{16}^j)\bigg).
\end{eqnarray*}\end{Exaexa}
\begin{Exaexa} Let $q=53=4\cdot 13+1$. Then $s=2$,  $t=13$.  By Theorem \ref{os1},   $\alpha_4=\sqrt{-1}=\sqrt{52}=2\sqrt{13}=2\cdot 16^3=30$. Using Theorem \ref{fact2nt4t+1}, the factorization of $x^{2^n\cdot13}-1$ over $\mathbb{F}_{53}$ can be written as a product of  $26n$  irreducible factors as:
\begin{eqnarray*}\label{ex2nd4k+1}
x^{2^n\cdot13}-1=\prod_{j=0}^{12}\bigg((x\pm{16}^j)(x\pm30\cdot{16}^j)\prod_{\substack{1\le r\le n-2}}(x^{2^r}\pm30\cdot{16}^j)\bigg).
\end{eqnarray*}\end{Exaexa}
\begin{Exaexa}  Let $q=59=2\cdot 29+1$. Then $t=29$ and $u=3$. Then $\mathbb{T}(\beta_4)=0$ and $\mathbb{T}(\beta_8)=\sqrt{-2}=(-2)^{15}=36$.  Then by Theorem \ref{phi2nd4t+1}, the factorization $\Phi_{464}(x)$ over $\mathbb{F}_{59}$ is given by
\begin{eqnarray*}\Phi_{464}(x)=
\displaystyle{\prod_{j=1}^{28} (x^{4}\pm 4^j\cdot{36}x^{2}-{16}^j)}
\end{eqnarray*} with $56$ irreducible trinomials over $\mathbb{F}_{59}$. \end{Exaexa}

\end {document}